\documentclass{amsart}

\usepackage[utf8]{inputenc}

\usepackage{graphicx}
\usepackage{amsmath,amsfonts,amssymb,amsthm}
\usepackage{hyperref}
\usepackage[capitalize]{cleveref}
\usepackage{thmtools}
\usepackage{xcolor}
\usepackage{subcaption}
\usepackage{float}
\usepackage{ytableau}
\usepackage{diagbox}
\usepackage{tikz}
\usetikzlibrary{external}
\graphicspath{ {./figures/} }

\theoremstyle{definition}

\newtheorem{thm}{Theorem}

\newtheorem{defn}[thm]{Definition}

\newtheorem{lem}[thm]{Lemma}
\newtheorem{prop}[thm]{Proposition}

\numberwithin{thm}{section}
\numberwithin{defn}{section}
\numberwithin{rem}{section}
\numberwithin{conj}{section}
\numberwithin{lem}{section}
\numberwithin{cor}{section}
\numberwithin{example}{section}
\numberwithin{prop}{section}
\numberwithin{figure}{section}
\numberwithin{table}{section}
\numberwithin{equation}{section}

\author{D Chen}
\title{Alternative Proof of the Determinant of Complete Non-Ambiguous Trees}

\begin{document}
\maketitle
\begin{abstract}
    Complete non-ambiguous trees have been studied in various contexts. Recently, a conjecture was made about their determinants, and subsequently proved by Aval. An alternative proof is given here.
\end{abstract}

\section{Introduction}

Complete non-ambiguous trees (CNATs) are a combinatorial object of recent interest \cite{abs10}. They may be seen as the proper way to embed a complete binary tree into a grid (see \cref{def:CNAT}).
Recently, a study of these objects was launched in \cite{co24}, in which a conjecture was made that the determinant is evenly distributed for odd-size CNATs. This was subsequently proved bijectively in \cite{a23}.

In this short article, first some self-contained context to the problem is given by stating the background definitions and conjecture. Then an alternative proof is given.

\section{Definitions}
    
The definition of a (complete) non-ambiguous tree is recalled from \cite{abs10} as follows:

\begin{defn}\label{def:CNAT}
    A \emph{non-ambiguous tree} is a filling of a rectangular grid in which each cell may or may not contain a dot, such that the following constraints are satisfied:

\begin{itemize}
    \item[] (\textbf{Existence of a root}) The top-left cell (the \emph{root}) is dotted.
    \item[] (\textbf{Non-ambiguity}) For every dotted cell (\emph{vertex}) other than the root, there is another vertex either in the same column and above, or in the same row and to the left, but not both.
    \item[] (\textbf{Minimality}) Every row and column contains at least one vertex.
\end{itemize}

By connecting each vertex to its unique precursor in the same row or column, a rooted binary tree is formed, thus motivating the name \emph{non-ambiguous tree}. If the underlying binary tree is complete, then the object is referred to as a \emph{complete non-ambiguous tree} (CNAT).

This implies that a CNAT with \begin{math}2n-1\end{math} vertices has \begin{math}n\end{math} leaves and \begin{math}n-1\end{math} internal vertices. Furthermore, every row and column has exactly one leaf, thus a CNAT with \begin{math}2n-1\end{math} vertices is contained in an \begin{math}n \times n\end{math} grid. We refer to the CNAT as having size \begin{math}n\end{math}, or dimensions \begin{math}n \times n\end{math}.

\end{defn}

\begin{figure}[H]
    \centering
    \includegraphics[width=35mm]{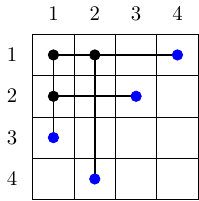}
    \caption{A 4x4 CNAT with leaf vertices highlighted in blue}
\end{figure}

\begin{defn}[Leaf Matrices]
A given CNAT, $M$, can be viewed as a square matrix by setting dotted cells to 1 and empty cells to 0. The \emph{leaf matrix} of $M$ is the matrix obtained when setting leaf vertices to 1, and all other cells to 0.
\end{defn}

Since there is exactly one leaf in every row and column, a leaf matrix is always a permutation matrix, and thus has determinant $\pm 1$. In fact, it was proved in Proposition 2.5 of \cite{co24} that the determinant of any CNAT is equal to the determinant of its leaf matrix.

\section{Statement of the conjecture}\label{sec:statement}

Given the preliminary definitions, we may now study the number of CNATs with determinant 1 and -1. Define these to be \begin{math}A_n\end{math} and \begin{math}B_n\end{math} respectively. Their sum is the total number of CNATs of a given size, \begin{math}T_n\end{math}. Also define the \emph{defect}, \begin{math}D_n = A_n - B_n\end{math}.

The numbers $T_n$ are well known, and Proposition 9 of \cite{abs10} gives a recurrence for them:
\begin{displaymath}
T_n = \sum_{k=1}^{n-1} {n-1 \choose k-1}{n-1 \choose k}T_kT_{n-k}
\end{displaymath}

Now, we may see the following computed data for $A_n$ and $B_n$.

\begin{center}
\begin{tabular}{ c|c|c|c|c|c|c|c|c } 
\begin{math}n\end{math} & 1 & 2 & 3 & 4 & 5 & 6 & 7 & 8\\
\hline
\text{\# with det 1, \begin{math}A_n\end{math}} & 1 & 0 & 2 & 17 & 228 & 4728 & 137400 & 5321889\\
\text{\# with det -1, \begin{math}B_n\end{math}} & 0 & 1 & 2 & 16 & 228 & 4732 & 137400 & 5321856\\
\hline
\begin{math}T_n = A_n+B_n\end{math} & 1 & 1 & 4 & 33 & 456 & 9460 & 274800 & 10643745\\
\begin{math}D_n = A_n-B_n\end{math} & 1 & -1 & 0 & 1 &  0 & -4 & 0 & 33\\
\end{tabular}
\end{center}

From the table, the determinant seems to be evenly distributed for odd \begin{math}n\end{math}; furthermore for even \begin{math}n\end{math} the defect \begin{math}D_n\end{math} seems to be an alternating version of \begin{math}T_n\end{math}. Indeed, this is the conjecture put forward by \cite{co24}:

\begin{displaymath} A_n-B_n = \begin{cases}
        1 & \text{if } n=1\\
        0 & \text{if } n>1 \text{ is odd}\\
        (-1)^{n/2}T_{n/2} & \text{if \begin{math}n\end{math} even}
    \end{cases} \end{displaymath}

In particular, since \begin{math}A_n + B_n = T_n\end{math} gives two simultaneous equations, it would imply that number of CNATs with determinants 1 and -1 can be written in terms of the sequence \begin{math}T\end{math}.

\section{Proof}

Finally, an inductive proof of the conjecture is given.

\begin{defn}
    For convenience, let \begin{math}e^n_k\end{math} be the number of \begin{math}k\end{math}-subsets of \begin{math}\{1, \dots, n\}\end{math} with an even sum. Similarly define \begin{math}o^n_k\end{math} for the number of subsets with odd sum.
\end{defn}

\begin{prop}\label{prop:recdet1} \begin{math}\forall n \geq 2\end{math},
\begin{equation*}
\begin{split}
    A_n = \sum_{k=1}^{n-1}
    & \left(A_kA_{n-k} + B_kB_{n-k})(e^{n-1}_{k-1}e^{n-1}_k + o^{n-1}_{k-1}o^{n-1}_k\right) \\
    & + \left(A_kB_{n-k} + B_kA_{n-k})(e^{n-1}_{k-1}o^{n-1}_k + o^{n-1}_{k-1}e^{n-1}_k\right)
\end{split}
\end{equation*}
\end{prop}

\begin{proof}
    Proposition 9 of \cite{abs10} establishes a recurrence for \begin{math}T_n\end{math}, derived by removing the root node of the CNAT to obtain two smaller CNATs. Let us extend their proof to formulate recurrences for \begin{math}A_n\end{math} and \begin{math}B_n\end{math}.

    By removing the root node of a CNAT, we obtain two smaller CNATs after flattening, whose rows and columns are interweaved together. Also, one contains all non-root vertices in the top row, and the other contains all non-root vertices in the leftmost column.
    Hence to construct a CNAT \begin{math}T\end{math} of size \begin{math}n\end{math} with determinant 1, we need to choose:
        \begin{enumerate}
            \item The size of the smaller CNAT, say \begin{math}M\end{math}, that contains the top row. Call this size \begin{math}k\end{math} (\begin{math}1 \leq k \leq n-1\end{math}). The other smaller CNAT, say \begin{math}M'\end{math}, contains the leftmost column and has size \begin{math}n-k\end{math}.
            
            \item The structures of both smaller CNATs. We can choose any combination of the two determinants, so long as in the next step we interweave the rows and columns so that the determinant is 1. We consider two cases; if \begin{math}\det M \det M' = 1\end{math} then there are \begin{math}A_kA_{n-k}+B_kB_{n-k}\end{math} ways since the determinants are both even or both odd, and similarly if \begin{math}\det M \det M' = -1\end{math} then there are \begin{math}A_kB_{n-k}+B_kA_{n-k}\end{math} ways.
            
            \item A way of interweaving the rows and columns so that the determinant becomes (or stays) 1. For this, we separately choose the rows of \begin{math}M\end{math} out of the rows of \begin{math}T\end{math}, and then the columns. For the rows, since \begin{math}M\end{math} must contain the top row, we have \begin{math}k-1\end{math} out of \begin{math}n-1\end{math} rows left to pick. For the columns, \begin{math}M\end{math} cannot contain the leftmost column because \begin{math}M'\end{math} does, so we have \begin{math}k\end{math} columns to pick out of \begin{math}n-1\end{math}.
            
            Now, we must pick the rows and columns so that the determinant is 1. Given a valid such choice, we may perform a sequence of swaps of adjacent rows or columns to turn it into a block diagonal matrix whose submatrices are \begin{math}M\end{math} and \begin{math}M'\end{math} (so that \begin{math}M\end{math} is contained in the first \begin{math}k\end{math} rows and first \begin{math}k\end{math} columns; \begin{math}T\end{math} need not still be a CNAT). The determinant of this block diagonal matrix will be \begin{math}\det M \det M'\end{math}. Note that we are only allowed to perform \emph{adjacent} swaps because the rows and columns in \begin{math}M\end{math} and \begin{math}M'\end{math} must stay in the same order. Each of these swaps will flip the sum of chosen rows and columns mod 2, and will also flip the determinant of \begin{math}T\end{math}. Hence if \begin{math}\det M \det M'\end{math} is 1 then the valid choices of rows and columns are precisely those where the sum of the rows and columns is congruent to \begin{math}(1+\dots+k)+(1+\dots+k) \equiv 0\end{math} mod 2. Thus there are \begin{math}e^{n-1}_{k-1}e^{n-1}_k + o^{n-1}_{k-1}o^{n-1}_k\end{math} ways, because the sum of rows and sum of columns must both be odd, or both be even. Similarly if \begin{math}\det M \det M'\end{math} is -1 then there are  \begin{math}e^{n-1}_{k-1}o^{n-1}_k + o^{n-1}_{k-1}e^{n-1}_k\end{math} ways.
        \end{enumerate}

        Summing the number of ways for each possibility of \begin{math}\det M \det M'\end{math}, then summing over \begin{math}k\end{math}, the result follows.
\end{proof}

\begin{lem}\label{lem:dets_diffrec}
\begin{equation*}
\begin{split}
    D_n = \sum_{k=1}^{n-1}
    & D_kD_{n-k}(e^{n-1}_{k-1} - o^{n-1}_{k-1})(e^{n-1}_k - o^{n-1}_k)
\end{split}
\end{equation*}
\end{lem}

\begin{proof}
    First note that $\forall n \geq 2$, the following formula holds:
        \begin{equation*}
        \begin{split}
            B_n = \sum_{k=1}^{n-1}
            & \left(A_kB_{n-k} + B_kA_{n-k})(e^{n-1}_{k-1}e^{n-1}_k + o^{n-1}_{k-1}o^{n-1}_k\right) \\
            & + \left(A_kA_{n-k} + B_kB_{n-k})(e^{n-1}_{k-1}o^{n-1}_k + o^{n-1}_{k-1}e^{n-1}_k\right)
        \end{split}
        \end{equation*}
    This can be seen in an identical way to the proof of \cref{prop:recdet1}, except in step 3 where we now wish to interweave the rows and columns so that the resulting determinant is -1.

    Now, subtract this expression for \begin{math}B_n\end{math} from the expression for \begin{math}A_n\end{math} in \cref{prop:recdet1}. The resulting equation factorises, and then occurrences of $A-B$ may be replaced with $D$.
\end{proof}

\begin{lem}\label{lem:dets_evenminusodd}
    \begin{math}\forall n\in \mathbb{N}, k \in \mathbb{Z}\end{math} with \begin{math}0 \leq k \leq n-1\end{math}:
    \begin{displaymath}e^{2n-1}_{2k} - o^{2n-1}_{2k} = (-1)^k {n-1 \choose k} \end{displaymath}
    \begin{displaymath}e^{2n-1}_{2k+1} - o^{2n-1}_{2k+1} = (-1)^{k+1} {n-1 \choose k} \end{displaymath}
\end{lem}

\begin{proof}
    Consider the group action of $C_2^k$ on subsets of ${1, \dots, 2n-1}$, with generators switching $1 \leftrightarrow 2, 3 \leftrightarrow 4, \dots, 2n-3 \leftrightarrow 2n-2$. If for some pair $(2t-1,2t)$ only one of them is in a subset, then the orbit of that subset has equally many subsets with even and odd sum because switching $2t-1 \leftrightarrow 2t$ changes the parity of the sum. The subsets left over are those in which both $2t-1$ and $2t$ are in the subset, or neither. Note that once we have chosen these pairs, there is no choice in whether to pick $2n-1$, because we must match the parity of the total number of items chosen. Thus there are ${n-1 \choose k}$ of the leftover subsets, and their sums all have the same parity. The factor of $(-1)^k$ follows.
\end{proof}

\begin{thm}
    \begin{displaymath}D_n = A_n-B_n = \begin{cases}
        1 & \text{if } n=1\\
        0 & \text{if } n>1 \text{ is odd}\\
        (-1)^{n/2}T_{n/2} & \text{if \begin{math}n\end{math} even}
    \end{cases}\end{displaymath}
\end{thm}

\begin{proof}
    Induction on $n$. Note it is true for $n=1,2$, and true for $n=3$ because $A_3=B_3=2$. Now suppose it is true up to $n-1$ where $n>3$; we will show it for $n$.

    If $n$ is odd:
    for all \begin{math}2 \leq k \leq n-2\end{math}, one of \begin{math}k\end{math} and \begin{math}n-k\end{math} is odd and larger than 1, because \begin{math}n\end{math} is odd. Hence \begin{math}D_kD_{n-k} = 0\end{math} by inductive hypothesis. Thus in the expression given for \begin{math}D_n\end{math} in \cref{lem:dets_diffrec}, all terms in the RHS sum vanish except \begin{math}k=1\end{math} and \begin{math}k=n-1\end{math}. However both of these terms still contribute zero, because \begin{math}n\end{math} is odd so \begin{math}e^{n-1}_1=o^{n-1}_1\end{math} (and also \begin{math}e^{n-1}_{n-2} = o^{n-1}_{n-2}\end{math} by taking the complement of subsets of size 1). So $D_n=0$ as required.

    Otherwise, $n$ is even: write \begin{math}n=2j, j \in \mathbb{N}\end{math}. In the expression given for $D_n$ in \cref{lem:dets_diffrec}, terms in the RHS sum vanish when $k$ is odd because $D_k = D_{2j-k} = 0$ by inductive hypothesis. Hence only the even terms contribute to the sum. Write \begin{math}k = 2r, r \in \mathbb{N}\end{math}:
    \begin{displaymath}D_{2j} = \sum_{r=1}^{j-1} D_{2r}D_{2j-2r}(e^{2j-1}_{2r-1} - o^{2j-1}_{2r-1})(e^{2j-1}_{2r} - o^{2j-1}_{2r})\end{displaymath}
    \begin{displaymath} = \sum_{r=1}^{j-1} D_{2r}D_{2j-2r}{j-1 \choose r-1}{j-1 \choose r} \;\text{by \cref{lem:dets_evenminusodd}}\end{displaymath}
    \begin{displaymath} = (-1)^j\sum_{r=1}^{j-1} T_rT_{j-r}{j-1 \choose r-1}{j-1 \choose r} \;\text{by inductive hypothesis}\end{displaymath}
    The RHS is \begin{math}(-1)^j\end{math} multiplied by the recurrence given in \cref{sec:statement} for the number of CNATs of a given size. Hence it equals \begin{math}(-1)^jT_j\end{math} as required.
    
\end{proof}

\bibliographystyle{alpha}
\bibliography{bibliography}

\begin{thebibliography}{ABBS14}

\bibitem[ABBS14]{abs10}
Jean-Christophe Aval, Adrien Boussicault, Mathilde Bouvel, and Matteo Silimbani.
\newblock {C}ombinatorics of non-ambiguous trees.
\newblock {\em Advances in Applied Mathematics}, 56:78--108, 2014.

\bibitem[Ava23]{a23}
Jean-Christophe Aval.
\newblock About the determinant of complete non-ambiguous trees, 2023.

\bibitem[CO24]{co24}
Daniel Chen and Sebastian Ohlig.
\newblock {Associated Permutations of Complete Non-Ambiguous Trees}.
\newblock {\em {Discrete Mathematics \& Theoretical Computer Science}}, {vol. 25:2 }, April 2024.

\end{thebibliography}
\addcontentsline{toc}{section}{References}

\end{document}